\documentclass[11pt,a4paper]{amsart}
\usepackage[mathscr]{eucal}
\usepackage{amssymb}
\usepackage{amsbsy}
\usepackage{graphics,graphicx}
\usepackage[all,knot,poly]{xy}

\usepackage[colorlinks,
            linkcolor=black,
            anchorcolor=black,
            citecolor=blue]{hyperref}
\usepackage{mathrsfs}
\usepackage{stmaryrd}
\usepackage{cases}
\usepackage{amssymb}
\usepackage{amsmath}
\usepackage{amsfonts}
\usepackage{graphicx}
\usepackage{amsmath,amstext,amsbsy,amssymb}
\usepackage[table]{xcolor}
\usepackage{multirow}
\usepackage{cases}
\usepackage{color}

\allowdisplaybreaks[4]  

\numberwithin{equation}{section}

\newcommand{\al}{\alpha}
\newcommand{\be}{\beta}
\newcommand{\de}{\delta}
\newcommand{\De}{\Delta}

\newcommand{\ga}{\gamma}
\newcommand{\Ga}{\Gamma}
\newcommand{\la}{\lambda}

\newcommand{\La}{\Lambda}
\newcommand{\ot}{\otimes}

\newcommand{\Om}{\Omega}

\newcommand{\te}{\theta}
\newcommand{\Te}{\Theta}
\newcommand{\vp}{\varphi}
\newcommand{\vt}{\vartheta}

\newcommand{\mB}{\mathcal{B}}

\newcommand{\mP}{\mathcal{P}}

\newcommand{\mc}{\mathscr{C}}
\newcommand{\pp}{\mathscr{P}}

\newcommand{\fs}{\mathfrak{S}}

\newcommand{\bN}{\mathbb{N}}

\newcommand{\bZ}{\mathbb{Z}}
\newcommand{\sh}{~\raisebox{0.55em}{\rotatebox{-90}{$\exists$}}~}

\newcommand{\mb}[1]{\mbox{#1}}

\newcommand{\ms}[1]{\mbox{\sffamily #1}}
\newcommand{\bs}[1]{{\scriptsize\mbox{#1}}}

\newcommand{\s}[1]{{\scriptsize #1}}
\newcommand{\ti}[1]{{\mbox{\tiny #1}}}
\newcommand{\stt}[1]{{\scriptstyle #1}}

\newcommand{\ol}[1]{\overline{#1}}

\newcommand{\us}[1]{\mbox{\upshape #1}}
\newcommand{\lan}{\left\langle}
\newcommand{\ran}{\right\rangle}
\newcommand{\lb}{\left(}
\newcommand{\rb}{\right)}
\newcommand{\rw}{\rightarrow}
\newcommand{\beq}{\begin{equation}}
\newcommand{\eeq}{\end{equation}}

\begin{document}

\newtheorem{theorem}{Theorem}[section]

\newtheorem{lem}[theorem]{Lemma}

\newtheorem{cor}[theorem]{Corollary}
\newtheorem{prop}[theorem]{Proposition}

\theoremstyle{remark}
\newtheorem{rem}[theorem]{Remark}

\newtheorem{defn}[theorem]{Definition}

\newtheorem{exam}[theorem]{Example}

\theoremstyle{conjecture}
\newtheorem{con}[theorem]{Conjecture}

\renewcommand\arraystretch{1.2}

\title[The shifted Poirier-Reutenauer algebra]{The shifted Poirier-Reutenauer algebra}

\author[Jing]{Naihuan Jing}
\address{Department of Mathematics, North Carolina State University, Raleigh, NC 27695-8205, USA}
\email{jing@math.ncsu.edu}

\author[Li]{Yunnan Li$^\ast$}
\address{Department of Mathematics, South China University of Technology, Guangzhou 510640, China}
\email{scynli@scut.edu.cn}

\thanks{$\ast$ corresponding author.}

\date\today

\begin{abstract}
Based on the shifted Schensted correspondence and the shifted Knuth equivalence,
 a shifted analog of the Poirier-Reutenauer algebra as a higher lift of Schur's P-functions and a right coideal subalgebra of the Poirier-Reutenauer algebra is constructed. Its close relations with
 the peak subalgebra and the Stembridge algebra of peak functions are also uncovered.

\end{abstract}

\maketitle

\section{Introduction}

Gessel \cite{Ge} used Stanley's theory of P-partitions to define the quasisymmetric functions $\ms{QSym}$ as  nonsymmetric generalization of symmetric functions. Stembridge \cite{Ste} extended Gessel's approach to define enriched P-partitions in order to give a combinatorial theory for Schur's Q-functions, which give rise to the peak functions positively refining Schur's Q-functions. On the other hand, the graded Hopf dual of the Stembridge algebra of peak functions is the peak subalgebra $\ms{Peak}$ in Solomon's descent algebra of the symmetric group.

In \cite{MR} Malvenuto and Reutenauer endowed the free module $\bZ\fs$ generated by permutations with two dual graded Hopf algebra structures and proved that $\ms{QSym}$ is a graded Hopf dual to Solomon's descent algebra, which is also isomorphic to the algebra $\ms{NSym}$ of noncommutative symmetric functions studied in \cite{GKL}. Later, Poirier and Reutenauer \cite{PR} defined two dual Hopf algebra structures $\ms{PR}$ and $\ms{PR}'$ from $\bZ\fs$ by passing to plactic classes, which have close relations with the Solomon descent algebra, symmetric and quasisymmetric functions (see diagrams (\ref{dia}),(\ref{dia1})).

The peak algebra $\ms{Peak}$ considered here is the Hopf algebra $\stackrel{\circ}{\mathfrak{P}}$ of interior peaks in \cite{ABN}, where Aguiar et al. defined a graded Hopf subalgebra $I^0$ of the algebra $\mb{\itshape Sol}^{\pm}(B)$ of noncommutative symmetric functions of type B, which is naturally projected onto $\stackrel{\circ}{\mathfrak{P}}$. As a lift of the embedding of the Solomon descent algebra $\mb{\itshape Sol}(A)$ of type A into the Malvenuto-Reutenauer algebra $\bZ\fs$, $\mb{\itshape Sol}^{\pm}(B)$ is embedded into the Hopf algebra $\bZ\mathfrak{B}$ of signed permutations. That is,

\[\raisebox{2em}{\xymatrix@H=1.5em{
I^0\,\,\ar@{>->}[r]\ar@{>>}[d]^-\vp& \mb{\itshape Sol}^{\pm}(B)\,\, \ar@{>->}[r]\ar@{>>}[d]^-\vp& \bZ\mathfrak{B}\ar@{>>}[d]^-\vp\\
\stackrel{\circ}{\mathfrak{P}}\,\,\ar@{>->}[r]& \mb{\itshape Sol}(A) \,\,\ar@{>->}[r]&\bZ\fs
}},\]
where $\vp$ is the map forgetting the signs of signed permutations.

In this paper, we only concentrate on the circumstance of
type A. That is, the shifted analog is constructed
 inside the Malvenuto-Reutenauer algebra $\bZ\fs$ of permutations. Recently Schur P-functions have been successively lifted onto the peak algebra $\ms{Peak}$ in \cite{JL}.
 Our current construction gives a higher lift of Schur P-functions onto the Poirier-Reutenauer algebra $\ms{PR}'$ as shown in diagram (\ref{dia'}). More explicitly, we derive a right coideal subalgebra $\ms{SPR}'$ of $\ms{PR}'$ with a natural basis indexed by marked shifted standard tableaux, which is projected onto that of
 Schur's P-functions. Meanwhile, the multiplication rule of $\ms{SPR}'$ directly implies the shifted Littlewood-Richardson rule. In the dual picture
  the canonical projection from $\ms{PR}$
to $\ms{SPR}$ lifts Stembridge's descent-to-peak map $\vt$, while the descent-to-peak transform $\Te:\ms{NSym}\rw\ms{Peak}$ (first defined in \cite{KLT}) has also been extended to the PR algebra level in the last part. As a result, we obtain a combinatorial expansion formula of the modified Schur functions in Schur's Q-functions (see (\ref{SQ})).

In order to define the shifted analog of the PR algebras, we consider the shifted plactic classes based on the shifted Knuth equivalence due to Sagan \cite{Sag}. It is worthy mentioning that Serrano \cite{Ser} also realized Schur's P-functions inside the shifted plactic algebra based on Haiman's mixed insertion \cite{Hai},
and that has been proved to be dual to the shifted Schensted correspondence given in \cite{Sag}. From this point of view, our approach can be partly thought as a dual version of Serrano's work. Nevertheless, our picture
 delivers more properties in the sense that the shifted Poirier-Reutenauer algebras are proved
 to carry intrinsic relations with both the peak algebra and the Stembridge algebra of peak functions. In order to reveal their relationship, we need to study the combinatorics of peaks for permutations and shifted tableaux in the context of the shifted Schensted correspondence. It is also noted that the analog of PR algebras in the type B case (inside the Hopf algebra of signed permutations) has been found by Baumann and Hohlweg in \cite[\S 5]{BH}, based on the Robinson-Schensted-Okada correspondence. Their work also revealed the relation with the algebra of (quasi)symmetric functions of type B.

The organization of the paper is summarized as follows. In $\S 2$ we provide the notation and definitions for some basic combinatorial objects. In $\S 3$ we recall the shifted Schensted correspondence and prove a key lemma about the descents and peaks of tableaux. In $\S 4$ our main construction is given. First we introduce combinatorial Hopf algebras involved in the construction, then bring in the shifted analogs of Poirier-Reutenauer algebras and describe their structural maps. Finally, we prove that the shifted Poirier-Reutenauer algebra is a lift of the Schur P-function and
discuss its relation with other combinatorial Hopf algebras. In the end, two further questions are also raised.

\section{Background}

\subsection{Notation and definitions}
Denote by $\bN$ (resp. $\bN_0$) the set of positive (resp. nonnegative) integers. Given any $m,n\in\bN,\,m\leq n$, let $[m,n]:=\{m,m+1,\dots,n\}$ and $[n]:=[1,n]$ for short. Let $\mc(n)$ be the set of compositions of $n$, consisting of ordered tuples of positive integers summed up to $n$ and $\mc=\bigcup\limits_{n\geq1}^.\mc(n)$. Write $\al\vDash n$ when $\al\in \mc(n)$. Let $\pp(n)$ be the set of partitions of $n$, consisting of compositions with weakly decreasing parts, and $\pp=\bigcup\limits_{n\geq1}^.\pp(n)$. Write $\al\vdash n$ when $\al\in \pp(n)$. A partition is strict, if all of its parts are distinct.

Given $\al=(\al_1,\dots,\al_r)\vDash n$, let its length $\ell(\al):=r$ and define its associated \textit{descent set}
\[D(\al)=\{\al_1,\al_1+\al_2,\dots,\al_1+\cdots+\al_{r-1}\}
\subseteq[n-1].\]
This gives a bijection between the compositions of $n$
and  the subsets of $[n-1]$. Also, the refining order $\preceq$ on $\mc(n)$ is defined by
\[\al\preceq\be\mb{ if and only if }D(\be)\subseteq D(\al),\,\forall\al,\be\vDash n.\]
For any $D\subseteq[n-1]$, we define its \textit{peak set} by
\[\mb{Peak}(D)=\{i\in D\backslash\{1\}\,:\,i-1\notin D\}.\]
Let $\mb{Peak}(\al):=\mb{Peak}(D(\al))$. The peak sets in $[n-1]$ are those subsets of $[2,n-1]$ without consecutive numbers.

In general, let $\mc_w(n)$ be the set of \textit{weak} compositions of $n$, consisting of ordered tuples of nonnegative integers summed up to be $n$. All the notation for compositions can be used for the weak ones. For $\al\in\mc_w$, let $\ell(\al)=|\{i\,:\,\al_i>0\}|$.

Given a strict $\la=(\la_1,\dots,\la_r)\vdash n$, the \textit{shifted shape} of $\la$ is an array of boxes in which
the $i$-th row has $\la_i$ boxes, and is shifted $i-1$ units to the right with respect to the top row. A \textit{(marked) shifted tableau} $T$ of shape $\la$ is a filling of the shifted shape $\la$
with letters from the alphabet $X'=\{1'<1<2'<2<\cdots\}$ such that:

$\bullet$ rows and columns of $T$ are weakly increasing;

$\bullet$ each $k$ appears at most once in every column;

$\bullet$ each $k'$ appears at most once in every row;

$\bullet$ there are no primed entries on the main diagonal.

A shifted tableau of shape $\la$ is called \textit{standard} if it contains each of the entries $1,\dots,|\la|$ exactly once, and
\textit{marked-standard} if these entries are allowed to be primed.
We denote by $\mb{ShSSYT}^\pm(\la)$ the set of marked shifted tableaux of shape $\la$, $\mb{ShSYT}(\la)$ and $\mb{ShSYT}^\pm(\la)$ for the standard and the marked-standard one respectively. Given a marked shifted tableau $T$, let $|T|$ be the unmarked tableau obtained by removing all the primes of entries of $T$. The \textit{weight} of a (marked) shifted tableau $T$ is a weak composition $\mb{wt}(T):=(\al_1,\al_2,\dots)$,
where $\al_i$ is the multiplicity of $i$ and $i'$ in $T$ for all $i\geq1$. The \textit{reading word} of a (marked) shifted tableau $T$, denoted by $w(T)$, is the word obtained by reading the letters of $T$ row by row
upward from the bottom to the top
and from the left to the right along each row. The \textit{standardization} $\mb{st}(T)\in\mb{ShSYT}^\pm(\la)$ of a marked shifted tableau $T$ of shape $\la$ is defined in terms of the total order of $X'$. It just refines the order of entries of $T$  by reading those $i$'s in $T$ from left to right and those $i'$'s in $T$ from top to bottom.

For example, let
$T=\raisebox{1.6em}{\xy 0;/r.16pc/:
(0,0)*{};(20,0)*{}**\dir{-};
(0,-5)*{};(20,-5)*{}**\dir{-};
(5,-10)*{};(20,-10)*{}**\dir{-};
(10,-15)*{};(15,-15)*{}**\dir{-};
(0,0)*{};(0,-5)*{}**\dir{-};
(5,0)*{};(5,-10)*{}**\dir{-};
(10,0)*{};(10,-15)*{}**\dir{-};
(15,0)*{};(15,-15)*{}**\dir{-};
(20,0)*{};(20,-10)*{}**\dir{-};
(2.5,-2.5)*{\stt{1}};(7.5,-2.5)*{\stt{3'}};(12.5,-2.5)*{\stt{4'}};(17.5,-2.5)*{\stt{4}};
(7.5,-7.5)*{\stt{3'}};(12.5,-7.5)*{\stt{4}};
(17.5,-7.5)*{\stt{6}};(12.5,-12.5)*{\stt{6}};
\endxy}~$. The reading word $w(T)=63'4613'4'4$, the weight $\mb{wt}(T)=(1,0,2,3,0,2)$ and
the standardization
$\mb{st}(T)=\raisebox{1.6em}{\xy 0;/r.16pc/:
(0,0)*{};(20,0)*{}**\dir{-};
(0,-5)*{};(20,-5)*{}**\dir{-};
(5,-10)*{};(20,-10)*{}**\dir{-};
(10,-15)*{};(15,-15)*{}**\dir{-};
(0,0)*{};(0,-5)*{}**\dir{-};
(5,0)*{};(5,-10)*{}**\dir{-};
(10,0)*{};(10,-15)*{}**\dir{-};
(15,0)*{};(15,-15)*{}**\dir{-};
(20,0)*{};(20,-10)*{}**\dir{-};
(2.5,-2.5)*{\stt{1}};(7.5,-2.5)*{\stt{2'}};(12.5,-2.5)*{\stt{4'}};(17.5,-2.5)*{\stt{6}};
(7.5,-7.5)*{\stt{3'}};(12.5,-7.5)*{\stt{5}};
(17.5,-7.5)*{\stt{8}};(12.5,-12.5)*{\stt{7}};
\endxy}~$.

\section{The shifted Schensted correspondence}
In \cite[Th. 3.1]{Sag}, Sagan provided the \textit{shifted Schensted correspondence} between permutations $w\in\fs_n$ and pairs $(P, Q)$ of standard shifted tableaux of size $n$ with $Q$ marked. He then generalized it to the shifted RSK correspondence. Meanwhile, Worley independently obtained such a bijection in his Ph.D. thesis, which is
now called the \textit{Sagan-Worley insertion}. Our construction of shifted PR algebras
will be based on Sagan-Worley correspondence.

Now let us recall the shifted Schensted correspondence. Suppose $w=w_1\dots w_n\in\fs_n$. We recursively construct a sequence $(P_0,Q_0),\dots,(P_n,Q_n)=(P,Q)$ of tableaux, where
each $P_i$ ia a standard shifted tableau and $Q_i$ is a marked-standard shifted tableau of the same shape as $P_i$, as follows.
Set $(P_0,Q_0)=(\emptyset,\emptyset)$. For $i=1,\dots,n$, insert $w_i$ into $P_{i-1}$ in the following manner:

\smallskip
Insert $w_i$ into the first row, bumping out the smallest element $x$ that is strictly greater than $w_i$ (like the usual Schensted row insertion). This is continued according to two rules:

\smallskip
(1) if $x$ is not on the main diagonal, then it is inserted to the next row as explained above;

\smallskip
(2) if $x$ is on the main diagonal, then it sets out a chain of column insertions
as follows. First it is inserted to the
next column on the right by bumping out the smallest element $y$ that is strictly greater than
$x$. Continue such column insertions until a letter is placed at the end of a column, without bumping out
a new element (i.e. no numbers in the column are strictly greater than the inserted one).

\smallskip
Procedure (2) is called a \textit{non-Schensted move}. When the above insertion process terminates, the resulting tableau is $P_i$.
The shapes of $P_{i-1}$ and $P_i$ differ by one box. To obtain $Q_i$, add that box to $Q_{i-1}$, and mark it with $i$
if the non-Schensted move is not involved in the process, otherwise write $i'$ on the box. Finally we call $P$ the insertion tableau and $Q$ the recording tableau, and denote them by $P_{\bs{SW}}(w)$ and $Q_{\bs{SW}}(w)$ respectively.

Recall that the shifted Knuth transformations for permutations in $\fs_n$ are (see \cite{Sag})

(SK1)\quad$xzy\sim zxy$ if $x<y<z$,

(SK2)\quad$yxz\sim yzx$ if $x<y<z$,

(SK3)\quad$xy\sim yx$ if $x,y$ are the first two letters of the permutation,

\noindent
which define the \textit{shifted Knuth equivalence} in $\fs_n$, denoted by $\equiv_{\bs{SK}}$. In \cite[Theorem 7.2]{Sag}, Sagan used the shifted Knuth equivalence to characterize permutations with the same insertion tableau $P$. Namely, for $u,v\in\fs_n$, $P_{\bs{SW}}(u)=P_{\bs{SW}}(v)$
if and only if $u\equiv_{\bs{SK}}v$.

The \textit{descent set} of a (marked-)standard shifted tableau $T$ of size $n$ is defined by
\[
\mb{Des}(T):=\left\{i\in[n-1]\,:\,
{i\mb{ strictly upper than }i+1\mb{ or } (i+1)'\mb{ in }T
\atop\mb{or }(i+1)'\mb{ weakly upper than }i \mb{ or } i' \mb{ in }T}\right\}.
\]
The associated composition of $\mb{Des}(T)$ is denoted by $c(T)$, and the \textit{peak set} of $T$ is defined by $\mb{Peak}(T):=\mb{Peak}(\mb{Des}(T))$. The \textit{descent set} of $w=w_1\dots w_n\in\fs_n$ is defined by
\[\mb{Des}(w):=\{i\in[n-1]\,:\,w_i>w_{i+1}\}.\]
The associated composition of $\mb{Des}(w)$ is denoted by $c(w)$, and the \textit{peak set} of $w$ is defined by $\mb{Peak}(w):=\mb{Peak}(\mb{Des}(w))$.

Next we give the following key lemma, which generalizes a result due to Sch\"{u}tzen\-berger \cite[\S 4, Ex. 17]{Ful}.

\begin{lem}\label{des}
For any permutation $w\in\fs_n$, we have

(1)\quad$\mb{Des}(Q_{\bs{SW}}(w))=\mb{Des}(w)$,

(2)\quad$\mb{Peak}(P_{\bs{SW}}(w^{-1}))=\mb{Peak}(w)$.
\end{lem}
\begin{proof} (1). Let $i\in[n-1]$ and suppose that $w_i<w_{i+1}$. After the insertion of $w_{i+1}$, we see that if $i$ is unprimed in $Q_{\bs{SW}}(w)$, then $i+1$ should also be unprimed and weakly upper than $i$. If both $i$ and $i+1$ are primed in $Q_{\bs{SW}}(w)$, then $(i+1)'$ should be strictly lower than $i'$ by the non-Schensted moves. Hence, it always means that $i\notin\mb{Des}(Q_{\bs{SW}}(w))$.

Now suppose that $w_i>w_{i+1}$. If $i$ is unprimed in $Q_{\bs{SW}}(w)$,
while $i+1$ is primed, then obviously $i\in\mb{Des}(Q_{\bs{SW}}(w))$ wherever $(i+1)'$ locates in $Q_{\bs{SW}}(w)$. If $i+1$ is also unprimed, then the row insertion makes $i$ to be strictly upper than $i+1$. Finally, if $i$ is primed in $Q_{\bs{SW}}(w)$, then $i+1$ should also be primed and weakly upper than $i'$. Hence, we always have $i\in\mb{Des}(Q_{\bs{SW}}(w))$.

(2). Given any standard shifted tableau $T$, $i\in\mb{Peak}(T)$ if and only if $i-1,i,i+1$ occur in $w(T)$ in the order $i-1,i+1,i$ or $i+1,i-1,i$. Meanwhile, the shifted Knuth transformations (SK1)-(SK3) can only possibly change one of these two orders of $i-1,i,i+1$ in $w$ to the other. That is, $\mb{Peak}(T)$ can be read from any permutation $w$ such that $P_{\bs{SW}}(w)=T$, as $w\equiv_{\bs{SW}}w(T)$. On the other hand, $i-1,i,i+1$ occur in $w^{-1}$ in the order $i-1,i+1,i$ or $i+1,i-1,i$ if and only if $w_{i-1}<w_i>w_{i+1}$, i.e. $i\in\mb{Peak}(w)$. Hence, $\mb{Peak}(P_{\bs{SW}}(w^{-1}))=\mb{Peak}(w)$.
\end{proof}

\begin{rem}\label{rem}
Haiman \cite{Hai} also introduced the mixed insertion ($w\mapsto(P_{\bs{mix}}(w),Q_{\bs{mix}}(w))$) and proved that it is dual to the Sagan-Worley insertion. That is, \[P_{\bs{mix}}(w)=Q_{\bs{SW}}(w^{-1}),\,
Q_{\bs{mix}}(w)=P_{\bs{SW}}(w^{-1})\]
for any permutation $w$. It generalizes the symmetry of $(P(w),Q(w))$ in the usual Schensted correspondence. Moreover, since the shifted plactic equivalence refines the plactic one \cite[Prop. 1.8]{Ser}, $P_{\bs{mix}}(w)=P_{\bs{mix}}(w')$ means $P(w)=P(w')$.
\end{rem}

\begin{exam}
For $w=612543\in\fs_6$, the shifted Schensted insertion makes
\[
\begin{array}{lllllll}
  P:\quad& {\xy 0;/r.2pc/:
(0,0)*{6};
\endxy}\,,\quad&
{\xy 0;/r.2pc/:
(0,0)*{1};(5,0)*{6};
\endxy}\,,\quad&
\raisebox{1em}{\xy 0;/r.2pc/:
(0,0)*{1};(5,0)*{2};(5,-5)*{6};
\endxy}\,,\quad&
\raisebox{1em}{\xy 0;/r.2pc/:
(0,0)*{1};(5,0)*{2};(10,0)*{5};(5,-5)*{6};
\endxy}\,,\quad&
\raisebox{1em}{\xy 0;/r.2pc/:
(0,0)*{1};(5,0)*{2};(10,0)*{4};(5,-5)*{5};(10,-5)*{6};
\endxy}\,,\quad&
\raisebox{1em}{\xy 0;/r.2pc/:
(0,0)*{1};(5,0)*{2};(10,0)*{3};(15,0)*{6};(5,-5)*{4};(10,-5)*{5};
\endxy}\vspace{1em}\\
Q:\quad& {\xy 0;/r.2pc/:
(0,0)*{1};
\endxy}\,,\quad&
{\xy 0;/r.2pc/:
(0,0)*{1};(5,0)*{2'};
\endxy}\,,\quad&
\raisebox{1em}{\xy 0;/r.2pc/:
(0,0)*{1};(5,0)*{2'};(5,-5)*{3};
\endxy}\,,\quad&
\raisebox{1em}{\xy 0;/r.2pc/:
(0,0)*{1};(5,0)*{2'};(10,0)*{4};(5,-5)*{3};
\endxy}\,,\quad&
\raisebox{1em}{\xy 0;/r.2pc/:
(0,0)*{1};(5,0)*{2'};(10,0)*{4};(5,-5)*{3};(10,-5)*{5'};
\endxy}\,,\quad&
\raisebox{1em}{\xy 0;/r.2pc/:
(0,0)*{1};(5,0)*{2'};(10,0)*{4};(15,0)*{6'};(5,-5)*{3};(10,-5)*{5'};
\endxy}
\end{array}\]
On the other hand, $w^{-1}=236541$, and
\[
\begin{array}{lllllll}
  P:\quad& {\xy 0;/r.2pc/:
(0,0)*{2};
\endxy}\,,\quad&
{\xy 0;/r.2pc/:
(0,0)*{2};(5,0)*{3};
\endxy}\,,\quad&
\xy 0;/r.2pc/:
(0,0)*{2};(5,0)*{3};(10,0)*{6};
\endxy\,,\quad&
\raisebox{1em}{\xy 0;/r.2pc/:
(0,0)*{2};(5,0)*{3};(10,0)*{5};(5,-5)*{6};
\endxy}\,,\quad&
\raisebox{1em}{\xy 0;/r.2pc/:
(0,0)*{2};(5,0)*{3};(10,0)*{4};(5,-5)*{5};(10,-5)*{6};
\endxy}\,,\quad&
\raisebox{1em}{\xy 0;/r.2pc/:
(0,0)*{1};(5,0)*{2};(10,0)*{3};(15,0)*{4};(5,-5)*{5};(10,-5)*{6};
\endxy}\vspace{1em}\\
Q:\quad& {\xy 0;/r.2pc/:
(0,0)*{1};
\endxy}\,,\quad&
{\xy 0;/r.2pc/:
(0,0)*{1};(5,0)*{2};
\endxy}\,,\quad&
\raisebox{1em}{\xy 0;/r.2pc/:
(0,0)*{1};(5,0)*{2};(10,0)*{3};
\endxy}\,,\quad&
\raisebox{1em}{\xy 0;/r.2pc/:
(0,0)*{1};(5,0)*{2};(10,0)*{3};(5,-5)*{4};
\endxy}\,,\quad&
\raisebox{1em}{\xy 0;/r.2pc/:
(0,0)*{1};(5,0)*{2};(10,0)*{3};(5,-5)*{4};(10,-5)*{5'};
\endxy}\,,\quad&
\raisebox{1em}{\xy 0;/r.2pc/:
(0,0)*{1};(5,0)*{2};(10,0)*{3};(15,0)*{6'};(5,-5)*{4};(10,-5)*{5'};
\endxy}
\end{array}\]
We can see that $\mb{Des}(Q_{\bs{SW}}(w))=\mb{Des}(w)=\{1,4,5\},
\,\mb{Des}(Q_{\bs{SW}}(w^{-1}))=\mb{Des}(w^{-1})=\{3,4,5\}$ and $\mb{Peak}(P_{\bs{SW}}(w^{-1}))=\mb{Peak}(w)=\{4\},\,
\mb{Peak}(P_{\bs{SW}}(w))=\mb{Peak}(w^{-1})=\{3\}$.
\end{exam}

\section{The shifted Poirier-Reutenauer algebra}
In this section we lift Schur's P-functions to the Malvenuto-Reutenauer algebra and define the so-called
shifted Poirier-Reutenauer algebra which contains the peak subalgebra. It is parallel to the construction of the Poirier-Reutenauer algebra, which is the lift of the Schur functions and contains the Solomon descent subalgebra.

\subsection{Preliminaries}
Gelfand et al. \cite{GKL} systematically studied the algebra of \textit{noncommutative symmetric functions}, denoted by $\ms{NSym}$.
This is a graded Hopf algebra freely generated by the $H_n, n\in\bN$ with the comultiplication given by
\beq\label{coq}
\De(H_n)=\sum_{k=0}^n H_k\ot H_{n-k},\eeq
where $H_0=1$ and the generator $H_n$ is of degree $n$.

Let \[H_\al=H_{\al_1}\cdots H_{\al_r},~\al=(\al_1,\dots,\al_r)\vDash n.\]
Then $\{H_\al\}_{\al\vDash n}$ forms a $\bZ$-basis of $\ms{NSym}_n$, called the {\it noncommutative complete symmetric functions}. There exists another important $\bZ$-basis $\{R_\al\}_{\al\vDash n}$ of $\ms{NSym}_n$, called the \textit{noncommutative ribbon Schur functions} defined by
\[R_\al=\sum_{\be\succeq\al}(-1)^{l(\be)-l(\al)}H_\be.\]

In \cite{MR} Malvenuto and Reutenauer proved that the graded Hopf dual of \ms{NSym} is the algebra \ms{QSym} of \textit{quasisymmetric functions}, which is a subring of the ring of power series $\bZ[[x_1,x_2,\dots]]$ in the commuting variables ${x_n}'s$.  \ms{QSym} has a linear basis called the \textit{monomial quasisymmetric functions}, defined by
\[M_\al:=M_\al(x)=\sum\limits_{i_1<\cdots< i_r}x_{i_1}^{\al_1}\cdots x_{i_r}^{\al_r},\]
where $\al=(\al_1,\dots,\al_r)$ varies over the composition set $\mc$. There is another important basis of \textit{fundamental quasisymmetric functions} defined by
\[F_\al:=F_\al(x)=\sum\limits_{i_1\leq\cdots\leq i_n\atop i_k<i_{k+1}\mb{ \tiny if }k\in D(\al)}x_{i_1}\cdots x_{i_n},\,\al\vDash n.\]
That means $F_\al=\sum_{\be\preceq\al}M_\be$. Meanwhile, the canonical pairing $\lan\cdot,\cdot\ran$ between \ms{NSym} and \ms{QSym} is defined by
\[\lan H_\al,M_\be\ran=\lan R_\al,F_\be\ran=\de_{\al,\be}\]
for any $\al,\be\in\mc$.

Furthermore, Malvenuto and Reutenauer defined a self-dual graded Hopf algebra structure on the free abelian
group $\mathbb{Z}\fs=\bigoplus_{n\geq0}\mathbb{Z}\fs_n$ of permutations, referred to as the \textit{MR algebra}.
We first recall the MR algebra, which is a dual pair of graded Hopf algebras $(\mathbb{Z}\mathfrak{S},*,\De)$
and $(\mathbb{Z}\mathfrak{S},*',\De')$, denoted $\ms{MR}$ and $\ms{MR}'$ respectively \cite{AgS,MR}.
For any word $w$ of length $n$ in a totally ordered alphabet $A$, denote by $\us{alph}(w)\subset A$ the set of letters in $w$ and $\us{st}(w)\in\mathfrak{S}_n$, the \textit{standardization} of $w$, which is
the well-defined permutation given by
\[\us{st}(w)(i)<\us{st}(w)(j)\mbox{ if and only if }w_i<w_j\mbox{ or }w_i=w_j,i<j,\]
where $w=w_1\cdots w_n$.
For any $w\in\mathfrak{S}_n$, one can view it as a word in $[n]$. For $I\subseteq[n]$, let $w|_I$ denote the subword of $w$ keeping only the digits in $I$. We also need the shuffle product \sh and the concatenation coproduct $\de'$. The product \sh is recursively defined by
\[u\sh\emptyset=\emptyset\sh u=u,\,u\sh v=u_1(u'\sh v)+v_1(u\sh v'),\]
for the words $u=u_1\cdots u_p,\,v=v_1\cdots v_q$, where $u'=u_2\cdots u_p,\,v'=v_2\cdots v_q$. The coproduct $\de'$ is defined by
\[\de'(w)=\sum_{i=0}^n w_1\cdots w_i\ot w_{i+1}\cdots w_n\]
for any word $w=w_1\cdots w_n$.

The Hopf algebra structure of the MR algebra can be described as follows. For any $w\in\mathfrak{S}_p,~w'\in\mathfrak{S}_q$,
\beq\label{mc}
w*w'=\sum\limits_{\us{\ti{alph}}(u)\cup\us{\ti{alph}}(v)=[p+q]
\atop\us{\ti{st}}(u)=w,~\us{\ti{st}}(v)=w'}uv,~
\De(w)=\sum\limits_{i=0}^n\us{st}(w|_{[1,i]})\ot\us{st}(w|_{[i+1,n]}).
\eeq
On the other hand,
\beq\label{mc'}
w*'w'=w\sh\ol{w'},~
\De'(w)=(\us{st}\ot\us{st})\de'(w),
\eeq
where $\ol{w'}$ means shifting the digits in $w'$ by $p$.

Under the canonical pairing of $\mathbb{Z}\mathfrak{S}$ defined by $\lan w,w'\ran=\de_{w,w'}$, $\ms{MR}$ and $\ms{MR}'$ are dual to each other as graded Hopf algebras.
Meanwhile, there exists an involution $\eta$ of graded Hopf algebras between them defined by $\eta(w)=w^{-1}$ and the embedding $\iota:\ms{NSym}\rightarrow \ms{MR}$ defined by $\iota(H_\al)=D_{\succeq\al}$ (resp. $\iota(R_\al)=D_\al$), where $D_{\succeq\al}=\sum_{c(w)\succeq\al}w$ (resp. $D_\al=\sum_{c(w)=\al}w$). This embedding identifies $\ms{NSym}$ with the \textit{Solomon descent subalgebra} of $\mathbb{Z}\mathfrak{S}$ with a basis $\{D_{\succeq\al}:\al\in\mc\}$.
Dually there exists a surjection $\pi':\ms{MR}'\rightarrow \ms{QSym}$ defined by $\pi'(w)=F_{c(w)}$. That is, $\lan\iota(F),w\ran=\lan F,\pi'(w)\ran$ for any $F\in\ms{NSym}$ and $w\in\fs$.

On the other hand, the descent algebra of $\mathbb{Z}\mathfrak{S}$ has a subalgebra $\ms{Peak}$, called the \textit{peak subalgebra}. It has a natural basis consisting of
\[\Pi_P:=\sum\limits_{\bs{Peak}(w)=P}w,\]
where $P$ runs over all peak sets. The graded Hopf dual $\ms{Peak}^*$ is isomorphic to the \textit{Stembridge algebra of peak functions} \cite{Ste}, realized as a Hopf subalgebra in $\ms{QSym}$ with a basis consisting of the \textit{peak functions}
\[K_P:=\sum_{\al\vDash n\atop P\subseteq D(\al)\cup(D(\al)+1)}2^{\ell(\al)}M_\al,\]
where $P$ varies over all peak sets in $[n-1]$. Also,
\beq\label{kf}K_P=2^{|P|+1}\sum_{\al\vDash n\atop P\subseteq D(\al)\triangle(D(\al)+1)}F_\al,\eeq
where $D\triangle(D+1)=D\backslash(D+1)\cup (D+1)\backslash D$ for any $D\subseteq[n-1]$.

The Hopf pairing $[\cdot,\cdot]$ between $\ms{Peak}$ and $\ms{Peak}^*$ is defined by
\[[\Pi_P,K_Q]=\de_{P,Q}\]
for peak sets $P,Q$. Moreover, there exists the \textit{descent-to-peak map} \cite{Ste}
\[\vt:\ms{QSym}\rw\ms{Peak}^*,\,F_\al\mapsto K_{\bs{Peak}(\al)}\]
as a Hopf algebra projection, which is dual to the \textit{descent-to-peak transform} $\Te$ from $\ms{NSym}$ onto $\ms{Peak}$ defined by (see \cite{KLT}, \cite{Sch})
\[\Te(H_\al)=2^{\ell(\al)}\sum_{P\ti{ peak set in }[n-1]\atop P\subseteq D(\al)\cup(D(\al)+1)}\Pi_P,\,\al\vDash n.\]
That is, $\lan\Te(F),f\ran=\lan F,\vt(f)\ran$ for any $F\in\ms{NSym},f\in\ms{QSym}$. By \cite[Prop. 5.5]{Sch}, we have
\beq \Te(R_\al)=\sum_{P\ti{ peak set in } [n-1]
\atop P\subseteq D(\al)\triangle(D(\al)+1)}2^{|P|+1}\Pi_P,\,\al\vDash n.\eeq

Now let $\La$ be the graded ring of symmetric functions in the commuting variables $x_1,x_2,\dots$, with integral coefficients,  and $\Om$ be the subring of $\La$ generated by the symmetric functions $q_n\,(n\geq1)$ defined by
\[\sum_{n\geq0}q_nz^n=\prod_{i\geq1}\dfrac{1+x_iz}{1-x_iz}.\]
For the basics of this subring $\Om$ and Schur's Q-functions, one can refer to \cite[Ch. III, \S 8]{Mac}, where $\Om$ is denoted as $\Ga$.
There exist two Hopf algebra epimorphisms. One is
\[\te:\La\rw\Om,\quad h_n\mapsto q_n,\,n\geq1\]
and the other is the \textit{forgetful map}
\[\phi:\ms{NSym}\rw\La,\,H_n\mapsto h_n,\,n\geq1\]
such that the following commutative diagrams hold:
\[\xymatrix@=2em{\mb{NSym}\ar@{->}[r]^-{\Te}\ar@{->}[d]_-{\phi}
&\mP\ar@{->}[d]^-{\phi}\\
\La\ar@{->}[r]^-{\te}&\Om},\quad
\xymatrix@=2em{\mb{QSym}\ar@{->}[r]^-{\vt}
&\mB\\
\La\ar@{->}[r]^-{\te}\ar@{->}[u]&\Om\ar@{->}[u]},\]
where the vertical maps in the second diagram are inclusions. Note that $\La$ and $\Om$ are self-dual graded Hopf algebras with the natural Hopf pairings $\lan\cdot,\cdot\ran$ and $[\cdot,\cdot]$ respectively. Then
\beq\label{bil}\begin{cases}
\,\lan\phi(F),f\ran=\lan F,f\ran,&F\in\ms{NSym},f\in\La,\\
\,[\phi(F),f]=[F,f],&F\in\ms{Peak},f\in\Om,\\
\,\lan F,f\ran=[\Te(F),f],&F\in\ms{NSym},f\in\ms{Peak}^*,\\
\,\lan F,f\ran=[F,\vt(f)],&F\in\ms{Peak},f\in\ms{QSym},\\
\,\lan \Te(F),f\ran=\lan F,\vt(f)\ran,&F\in\ms{NSym},f\in\ms{QSym}.
\end{cases}\eeq

\subsection{The shifted analog of Poirier-Reutenauer algebras}

Note that (SK1), (SK2) generate the usual \textit{Knuth equivalence}, denoted by $\equiv_{\bs{K}}$. Moreover, the $\bZ$-submodule $J_{\bs{K}}$ spanned by
\[\{u-v\,:\,u,v\mb{ permutations},\,u\equiv_{\bs{K}}v\}\]
is actually a Hopf ideal in $\ms{MR}$, which induces the Hopf quotient $\ms{PR}$, called the Poirier-Reutenauer algebra (PR for short) \cite{PR}. Dually, one can define a Hopf subalgebra $\ms{PR}'$ of $\ms{MR}'$ as the orthogonal complement of $J_{\bs{K}}$ with respect to the canonical pairing. That is,
\[\ms{PR}'=\bigoplus\limits_{T\in\us{\s{SYT}}}\bZ\cdot cl(T), ~cl(T)=\sum\limits_{w\equiv_{\ti{K}}w(T)}w=\sum\limits_{P(w)=T}w,\]
where $P(w)$ is the insertion tableau of $w$ under the Schensted insertion. The dual of $cl(T)$ in $\ms{PR}$ is the Knuth equivalence class $[w]$ such that $P(w)=T$, and we denote it by $[T]$.

A remarkable property of the PR algebra $\ms{PR}'$ is that it can factor through the embedding $\eta\circ\iota:\ms{NSym}\rightarrow\ms{MR}'$ and has the image $\La$ when projecting to $\ms{QSym}$ via $\pi'$. In fact, $\eta\circ\iota(R_\al)=\sum_{c(T)=\al}cl(T)$ for any $\al\vDash n$. Moreover,  $\pi'(cl(T)),\,T\in\mb{SYT}(\la)$, is the \textit{Schur function} $s_\la$ with the following combinatorial expression,
\[s_\la=\sum\limits_{T\in\us{\s{SSYT}}(\la)}x^{\us{\s{wt}}(T)}
=\sum\limits_{T\in\ti{SYT}(\la)}F_{c(T)}.\]
In summary, the following commutative  diagram of Hopf algebras holds (see also \cite[Sect. 1]{DHT}):
\begin{equation}\label{dia}
\xymatrix@H=1.5em{
\ms{NSym }\ar@(ur,ul)[rr]^-{\phi} \ar@{>->}[d]^-\iota \ar@{>->}[r]& \ms{PR}' \ar@{>>}[r]^-{\pi'}\ar@{>->}[d]&\La\ar@{>->}[d]\\
\ms{MR }\ar[r]^-{\eta\atop\sim} & \ms{MR}' \ar@{>>}[r]^-{\pi'}&\ms{QSym}\quad.
}
\end{equation}
The dual diagram is
\begin{equation}\label{dia1}
\xymatrix@H=1.5em{
\ms{NSym }\ar@{>->}[r]^-\iota \ar@{>>}[d]^-\phi&
\ms{MR }\ar@{>>}[d]\ar[r]^-{\eta\atop\sim}&\ms{MR}' \ar@{>>}[d]^-{\pi'}\\{\La\,}\ar@(dr,dl)[rr]
\ar@{>->}[r] & \ms{PR} \ar@{>>}[r]&\ms{QSym}\quad,}
\end{equation}

\medskip
\noindent
where the embedding of $\La$ into $\ms{PR}$ maps the Schur function $s_\la$ to $\sum_{T\in\bs{SYT}(\la)}[T]$, while the projection of $\ms{PR}$ onto $\ms{QSym}$ maps $[T]$ to $F_{c(T)}$.

Now we seek for the shifted analog of the above construction. First, it is natural to ask what structure does the $\bZ$-submodule $J_{\bs{SK}}$ of $\ms{MR}$ defined from shifted Knuth equivalence possess, where $J_{\bs{SK}}$ is spanned by
\[\{u-v\,:\,u,v\mb{ permutations},\,u\equiv_{\bs{SK}}v\}.\]

Let $\ms{SPR}:=\ms{MR}/J_{\bs{SK}}$, then we have the following result.

\begin{theorem}\label{spr}
$\ms{SPR}$ is a right $\ms{PR}$-module and a quotient coalgebra of $\ms{PR}$.
\end{theorem}
\begin{proof}
As $J_{\bs{K}}\subset J_{\bs{SK}}$, it is equivalent to prove that $J_{\bs{SK}}$ is a right ideal and a coideal of $\ms{MR}$. Since $J_{\bs{K}}$ is known to be a Hopf ideal of $\ms{MR}$, we only need to focus on the shifted Knuth transformation (SK3). For any word $w=w_1w_2\cdots w_p,\,p\geq2$, let $\tilde{w}:=w_2w_1w_3\cdots w_p$, then for any $w\in\fs_p,\,w'\in\fs_q,\,p\geq2$,
\[(w-\tilde{w})*w'=\sum\limits_{\us{\ti{alph}}(u)\cup\us{\ti{alph}}(v)=[p+q]
\atop\us{\ti{st}}(u)=w,~\us{\ti{st}}(v)=w'}(u-\tilde{u})v\in J_{\bs{SK}},\]
since obviously $\mb{st}(\tilde{u})=\widetilde{\mb{st}(u)}=\tilde{w}$. On the other hand, for any fixed $i=0,\dots,n$, if $w_1,w_2\in[1,i]$ (resp. $w_1,w_2\in[i+1,n]$), then $w|_{[1,i]}-\tilde{w}|_{[1,i]}$ (resp. $\mb{st}(w|_{[i+1,n]})-\mb{st}(\tilde{w}|_{[i+1,n]})$) lies in $J_{\bs{SK}}$. Otherwise, $w_1,w_2$ lie in $[1,i]$ and $[i+1,n]$ separately, hence $w|_{[1,i]}=\tilde{w}|_{[1,i]},\,w|_{[i+1,n]}=
\tilde{w}|_{[i+1,n]}$ by definition, and the term in $\De(w-\tilde{w})$ for such $i$ vanishes. In summary, we have $\De(w-\tilde{w})\in J_{\bs{SK}}$.
\end{proof}

 We remark that $w*(w'-\tilde{w'})\notin J_{\bs{SK}}$ in general, thus $\ms{SPR}$ is \textit{not} a Hopf quotient of $\ms{PR}$. A simple counterexample is the following:
\[\begin{array}{ccccccccccc}
12*123: & 12345 & 13245 & 14235 & 15234 & 23145 & 24135 & 25134 & 34125 & 35124 & 45123\vspace{.5em}\\
P_{\bs{SW}}: &\raisebox{.5em}{\xy 0;/r.1pc/:
(0,0)*{1};(5,0)*{2};(10,0)*{3};(15,0)*{4};(20,0)*{5};
\endxy} &
\raisebox{1em}{\xy 0;/r.1pc/:
(0,0)*{1};(5,0)*{2};(10,0)*{4};(15,0)*{5};(5,-8)*{3};
\endxy} &
\raisebox{1em}{\xy 0;/r.1pc/:
(0,0)*{1};(5,0)*{2};(10,0)*{3};(15,0)*{5};(5,-8)*{4};
\endxy}&
\raisebox{1em}{\xy 0;/r.1pc/:
(0,0)*{1};(5,0)*{2};(10,0)*{3};(15,0)*{4};(5,-8)*{5};
\endxy}&
\raisebox{.5em}{\xy 0;/r.1pc/:
(0,0)*{1};(5,0)*{2};(10,0)*{3};(15,0)*{4};(20,0)*{5};
\endxy} &
\raisebox{1em}{\xy 0;/r.1pc/:
(0,0)*{1};(5,0)*{2};(10,0)*{3};(15,0)*{5};(5,-8)*{4};
\endxy}&
\raisebox{1em}{\xy 0;/r.1pc/:
(0,0)*{1};(5,0)*{2};(10,0)*{3};(15,0)*{4};(5,-8)*{5};
\endxy} &
\raisebox{1em}{\xy 0;/r.1pc/:
(0,0)*{1};(5,0)*{2};(10,0)*{4};(15,0)*{5};(5,-8)*{3};
\endxy}&
\raisebox{1em}{\xy 0;/r.1pc/:
(0,0)*{1};(5,0)*{2};(10,0)*{4};(5,-8)*{3};(10,-8)*{5};
\endxy}&
\raisebox{1em}{\xy 0;/r.1pc/:
(0,0)*{1};(5,0)*{2};(10,0)*{3};(5,-8)*{4};(10,-8)*{5};
\endxy}\vspace{.5em}\\
12*213: & 12435 & 13425 & 14325 & 15324 & 23415 & 24315 & 25314 & 34215 & 35214 & 45213 \vspace{.5em}\\
P_{\bs{SW}}: &\raisebox{1em}{\xy 0;/r.1pc/:
(0,0)*{1};(5,0)*{2};(10,0)*{3};(15,0)*{5};(5,-8)*{4};
\endxy} &
\raisebox{1em}{\xy 0;/r.1pc/:
(0,0)*{1};(5,0)*{2};(10,0)*{4};(15,0)*{5};(5,-8)*{3};
\endxy} &
\raisebox{1em}{\xy 0;/r.1pc/:
(0,0)*{1};(5,0)*{2};(10,0)*{4};(15,0)*{5};(5,-8)*{3};
\endxy}&
\raisebox{1em}{\xy 0;/r.1pc/:
(0,0)*{1};(5,0)*{2};(10,0)*{4};(5,-8)*{3};(10,-8)*{5};
\endxy}&
\raisebox{.5em}{\xy 0;/r.1pc/:
(0,0)*{1};(5,0)*{2};(10,0)*{3};(15,0)*{4};(20,0)*{5};
\endxy}&
\raisebox{1em}{\xy 0;/r.1pc/:
(0,0)*{1};(5,0)*{2};(10,0)*{3};(15,0)*{5};(5,-8)*{4};
\endxy}&
\raisebox{1em}{\xy 0;/r.1pc/:
(0,0)*{1};(5,0)*{2};(10,0)*{3};(15,0)*{4};(5,-8)*{5};
\endxy} &
\raisebox{.5em}{\xy 0;/r.1pc/:
(0,0)*{1};(5,0)*{2};(10,0)*{3};(15,0)*{4};(20,0)*{5};
\endxy}&
\raisebox{1em}{\xy 0;/r.1pc/:
(0,0)*{1};(5,0)*{2};(10,0)*{3};(15,0)*{4};(5,-8)*{5};
\endxy}&
\raisebox{1em}{\xy 0;/r.1pc/:
(0,0)*{1};(5,0)*{2};(10,0)*{3};(15,0)*{5};(5,-8)*{4};
\endxy}
\end{array}\]
Since the insertion tableaux at the two rows are not the same, we know that $12*(123-213)\notin J_{\bs{SK}}$.


\begin{defn}
Let $\ms{SPR}':=J_{\bs{SK}}^\perp$, the orthogonal complement of $J_{\bs{SK}}$ with respect to the canonical pairing, then it is dually a right coideal subalgebra of $\ms{PR}'$ by Theorem \ref{spr}. More explicitly,
\[\ms{SPR}'=\bigoplus\limits_{T\in\us{\s{ShSYT}}}\bZ\cdot scl(T), ~scl(T)=\sum\limits_{w\equiv_{\ti{SK}}w(T)}w=\sum\limits_{P_{\ti{SW}}(w)=T}w.\]
as $P_{\ti{SW}}(w(T))=T$. We call $\ms{SPR}$ and $\ms{SPR}'$ the \textit{shifted Poirier-Reutenauer algebras}.
\end{defn}
The dual of $scl(T)$ in $\ms{SPR}$ is the shifted Knuth equivalence class $\lan w\ran$ such that $P_{\ti{SW}}(w)=T$, and we denote it by $\lan T\ran$. Meanwhile, $\ms{SPR}'$ naturally embeds into $\ms{PR}'$, since for any $T\in\mb{ShSYT}_n$,
\beq\label{ss}scl(T)=\sum\limits_{P_{\ti{SW}}(w)=T}w
=\sum\limits_U\lb\sum\limits_{P(w)=U}w\rb=\sum\limits_Ucl(U),\eeq
where $U$ runs over all standard tableaux of size $n$ such that $P_{\bs{SW}}(w(U))=T$.

\medskip
Next we consider all structure maps of $\ms{SPR}$ and $\ms{SPR}'$.
For these we also need the notion of \textit{rectification} of skew shifted tableaux \cite[\S 6]{Sag}. Given $S\in\us{ShSSYT}(\la/\mu)$, the rectification $\us{rect}(S)$ is the shifted tableau obtained from $S$ by carrying out Sch\"{u}tzenberger's \textit{slide} (or ``jeu de taquin'') repeatedly. Note that $w(\us{rect}(S))\equiv_{\bs{SK}}w(S)$ for any skew shifted tableau $S$. On the other hand, given $T\in\us{ShSYT}(\mu),~S\in\us{ShSYT}(\la/\mu)$ for $\mu\vdash n$, one shifts all entries of $S$ by $n$, and concatenates the resulting skew shifted tableau with $T$ naturally. It gives a standard shifted tableau of shape $\la$, denoted by $(T)_S$. For instance,
\[T=\raisebox{1.2em}{\xy 0;/r.16pc/:
(0,0)*{};(15,0)*{}**\dir{-};
(0,-5)*{};(15,-5)*{}**\dir{-};
(5,-10)*{};(10,-10)*{}**\dir{-};
(0,0)*{};(0,-5)*{}**\dir{-};
(5,0)*{};(5,-10)*{}**\dir{-};
(10,0)*{};(10,-10)*{}**\dir{-};
(15,0)*{};(15,-5)*{}**\dir{-};
(2.5,-2.5)*{\stt{1}};(7.5,-2.5)*{\stt{2}};
(12.5,-2.5)*{\stt{4}};(7.5,-7.5)*{\stt{3}};
\endxy}~,~S=\raisebox{1.6em}{\xy 0;/r.16pc/:
(5,0)*{};(10,0)*{}**\dir{-};
(0,-5)*{};(10,-5)*{}**\dir{-};
(0,-10)*{};(10,-10)*{}**\dir{-};
(0,-15)*{};(5,-15)*{}**\dir{-};
(0,-5)*{};(0,-15)*{}**\dir{-};
(5,0)*{};(5,-15)*{}**\dir{-};
(10,0)*{};(10,-10)*{}**\dir{-};
(7.5,-2.5)*{\stt{2}};(2.5,-7.5)*{\stt{1}};
(7.5,-7.5)*{\stt{4}};(2.5,-12.5)*{\stt{3}};
\endxy}~,~
(T)_S=\raisebox{1.6em}{\xy 0;/r.16pc/:
(0,0)*{};(20,0)*{}**\dir{-};
(0,-5)*{};(20,-5)*{}**\dir{-};
(5,-10)*{};(20,-10)*{}**\dir{-};
(10,-15)*{};(15,-15)*{}**\dir{-};
(0,0)*{};(0,-5)*{}**\dir{-};
(5,0)*{};(5,-10)*{}**\dir{-};
(10,0)*{};(10,-15)*{}**\dir{-};
(15,0)*{};(15,-15)*{}**\dir{-};
(20,0)*{};(20,-10)*{}**\dir{-};
(2.5,-2.5)*{\stt{1}};(7.5,-2.5)*{\stt{2}};(12.5,-2.5)*{\stt{4}};
(17.5,-2.5)*{\stt{6}};(7.5,-7.5)*{\stt{3}};(12.5,-7.5)*{\stt{5}};
(17.5,-7.5)*{\stt{8}};(12.5,-12.5)*{\stt{7}};
\endxy}~.\]

By formula (\ref{mc}) and Theorem \ref{spr}, we know that
\beq\label{mul}\lan w\ran*[w']=\sum\limits_{\us{\ti{alph}}(u)\cup\us{\ti{alph}}(v)=[p+q]
\atop\us{\ti{st}}(u)=w,~\us{\ti{st}}(v)=w'}\lan uv\ran,\,
\De(\lan w\ran)=\sum\limits_{i=0}^p\lan w|_{[1,i]}\ran\ot\lan\us{st}(w|_{[i+1,p]})\ran\eeq
for any $w\in\fs_p,\,w'\in\fs_q$.

For any $T\in\mb{ShSSYT}$ and $T'\in\mb{SSYT}$, we define
\[T\cdot T':=P_{\bs{SW}}(w(T)w(T')).\]
Now (\ref{mul}) can be transferred to the following formulas by the language of tableaux, that is,
\beq\label{mul1}\lan T_1\ran*[T_2]=\sum\limits_{T\cdot T'\in\ti{ShSYT}
\atop\us{\ti{st}}(T)=T_1,~\us{\ti{st}}(T')=T_2}\lan T\cdot T'\ran,\eeq
for any $T_1\in\mb{ShSYT}$, $T_2\in\mb{SYT}$, and
\beq\label{com}\De(\lan T\ran)=\sum\limits_{T=(T')_S}\lan T'\ran\ot\lan\mb{rect}(S)\ran.\eeq
for any $T\in\mb{ShSYT}$.

Dually, we have
\beq\label{pr'm}
scl(T_1)*' scl(T_2)=\sum\limits_{T=(T_1)_S \atop \us{\ti{rect}}(S)=T_2}scl(T).
\eeq
and
\beq\De'(scl(T))=\sum\limits_{T_1\in{\ti{ShSSYT}},\,T_2\in{\ti{SSYT}}\atop T=T_1\cdot T_2}scl(\mb{st}(T_1))\ot cl(\mb{st}(T_2)),\eeq

\subsection{$\ms{SPR}'$ as a lift of Schur's P-functions}
In this subsection, we convince the reader that $\ms{SPR}'$ is exactly the desired shifted analog of $\ms{PR}'$.

For any strict partition $\la$, \textit{Schur's P-function} $P_\la\in\Om$ is defined by \cite[\S III, ($8.16'$)]{Mac}
\[P_\la=\sum_{S\in\bs{ShSSYT}^\pm(\la)}x^{\us{\s{wt}}(S)}.\]
And \textit{Schur's Q-function} $Q_\la:=2^{\ell(\la)}P_\la$ so that $[P_\la,Q_\mu]=\de_{\la,\mu}$.

On the other hand, the following expansion formula is pointed out in \cite[Prop. 3.3]{As}.
\beq\label{pf}
P_\la=\sum_{S\in\bs{ShSYT}^\pm(\la)}F_{c(S)}.\eeq

In order to clarify formula (\ref{pf}), one needs the following fundamental bijection, as the shifted analog of that in \cite[Prop. 5.3.6]{Sag1}.
\begin{prop}\label{bij}
There exists a one-to-one correspondence between $\mb{ShSSYT}^\pm$ and
the set $\{(T,\ga)\in\mb{ShSYT}^\pm(\la)\times\mc_w\,:\,\ga\preceq c(T)\}$, mapping $S\in\mb{ShSSYT}^\pm$ to $(\mb{st}(S),\mb{wt}(S))$.
\end{prop}
\begin{proof}
First we check that for $S\in\mb{ShSSYT}^\pm$, $\mb{wt}(S)\preceq c(\mb{st}(S))$. That is, if $i\in\mb{Des}(\mb{st}(S))$, then when doing the standardization, the $i$th and $(i+1)$th numbers in $S$ should be distinct after removing their primes if necessary. Since $i\in\mb{Des}(\mb{st}(S))$ means one of the following two cases holds in $\mb{st}(S)$:

(1) $i$ is unprimed and strictly upper than $i+1$,

(2) $i+1$ is primed and weakly upper than $i$,

\noindent
it is easy to deduce the desired result by definition. Conversely, given a pair $(T,\ga)\in\mb{ShSYT}^\pm(\la)\times\mc_w$ such that $\ga\preceq c(T)$,
one can construct the marked shifted tableau backward.
\end{proof}

Now we are in the position to give our main result.
\begin{theorem}
The following shifted analog of diagram (\ref{dia}) holds:
\begin{equation}\label{dia'}
\raisebox{2em}{\xymatrix@H=1.5em{
\ms{Peak }\ar@(ur,ul)[rr]^-{\phi} \ar@{>->}[d]\ar@{>->}[r]& \ms{SPR}' \ar@{>>}[r]^-{\pi'}\ar@{>->}[d]&\Om\ar@{>->}[d]\\
\ms{NSym }\ar@{>->}[r]& \ms{PR}' \ar@{>>}[r]^-{\pi'}&\La
}},
\end{equation}
Dually, we have the following analogous diagram to (\ref{dia1}):
\begin{equation}\label{dia1'}
\raisebox{2em}{\xymatrix@H=1.5em{
{\La\,}\ar@{>->}[r]\ar@{>>}[d]^-\te&
\ms{PR }\ar@{>>}[d]\ar@{>>}[r]&\ms{QSym} \ar@{>>}[d]^-\vt\\{\Om\,\,}\ar@(dr,dl)[rr]
\ar@{>->}[r] & \ms{SPR} \ar@{>>}[r]&\ms{Peak}^*}},
\end{equation}

\medskip
\noindent
where all the objects in the two diagrams are Hopf algebras, except for the right coideal subalgebra $\ms{SPR}'$ of $\ms{PR}'$ and the right $\ms{PR}$-module coalgebra $\ms{SPR}$.
\end{theorem}

\begin{proof}
According to Lemma \ref{des} and formula (\ref{pf}), we know that
\beq\label{pro}
\pi'\lb scl(T)\rb=\sum_{P_{\ti{SW}}(w)=T}F_{c(w)}=
\sum_{P_{\ti{SW}}(w)=T}F_{c(Q_{\ti{SW}}(w))}
=\sum_{S\in\bs{ShSYT}^\pm(\la)}F_{c(S)}=P_\la
\eeq
for any $T\in\mb{ShSYT}(\la)$. That is, the projection $\pi'$ restricting on $\ms{SPR}'$ has the image $\Om$. On the other hand, for any peak set $P$ in $[2,n-1]$,
\[
\begin{split}
\eta\circ\iota(\Pi_P)&=\sum\limits_{\bs{Peak}(w)=P}w^{-1}
=\sum\limits_{\bs{Peak}(P_\ti{SW}(w^{-1}))=P}w^{-1}\\
&=\sum\limits_{T\in\ti{ShSYT}_n\atop\ti{Peak}(T)=P}
\lb\sum\limits_{P_{\ti{SW}}(w)=T}w\rb
=\sum\limits_{T\in\ti{ShSYT}_n\atop\ti{Peak}(T)=P}scl(T)\in\ms{SPR}'.
\end{split}
\]
Therefore, the embedding of $\ms{Peak}$ into $\ms{PR}'$ can factor through $\ms{SPR}'$.

In diagram (\ref{dia1'}), the embedding of $\Om$ into $\ms{SPR}$ maps Schur's Q-function $Q_\la$ to $\sum_{T\in\bs{ShSYT}(\la)}\lan T\ran$, while the projection of $\ms{SPR}$ onto $\ms{Peak}^*$ maps $\lan T\ran$ to $K_{\bs{Peak}(T)}$. In fact, denote the embedding by $i$ and the projection by $p$, then for any strict partition $\la$,
\[\begin{split}
i(Q_\la)&=\sum_{T\in\bs{ShSYT}}\lan i(Q_\la),scl(T)\ran\lan T\ran
=\sum_{T\in\bs{ShSYT}}[Q_\la,\pi'(scl(T))]\lan T\ran\\
&=\sum_{T\in\bs{ShSYT}}[Q_\la,P_{\bs{shape}(T)}]\lan T\ran
=\sum_{T\in\bs{ShSYT}(\la)}\lan T\ran.
\end{split}\]
On the other hand, note that for any $T\in\mb{SYT}$, $T=P(w(T))=Q(w(T)^{-1})$, thus \[\mb{Peak}(T)=\mb{Peak}(Q(w(T)^{-1}))=\mb{Peak}(w(T)^{-1})
=\mb{Peak}(P_{\bs{SW}}(w(T))\]
by Lemma \ref{des}, and the projection $\ms{PR}\twoheadrightarrow\ms{SPR}$ maps $[T]$ to $\lan P_{\bs{SW}}(w(T))\ran$ by definition. It induces the map $p:\ms{SPR}\rw\ms{Peak}^*$ such that $p(\lan S\ran)=K_{\bs{Peak}(S)}$ for any $S\in\mb{ShSYT}$. Moreover, $p\circ i$ is the identity map by the following expansion in \cite[(2.5)]{Ste}:
\beq\label{QK}Q_\la=\sum_{T\in\bs{ShSYT}(\la)}K_{\bs{Peak}(T)}.\eeq
\end{proof}

Combining (\ref{ss}) and (\ref{pro}), we have the well-known Schur P-positivity: For any strict $\la\vdash n$ and fixed $T\in\mb{ShSYT}(\la)$,
\beq\label{ps}P_\la=\pi'\lb scl(T)\rb=\sum\limits_U\pi'(cl(U))
=\sum\limits_{\mu\vdash n}|\{U\in\mb{SYT}(\mu)\,:\,P_{\bs{SW}}(w(U))=T\}|s_\mu.\eeq
Using the Robinson-Schensted correspondence and its symmetry property, we see that the set $\{U\in\mb{SYT}(\mu)\,:\,P_{\ti{SW}}(w(U))=T\}$ bijectively corresponds to
\beq\label{set}
\fs(T,V):=\{w\in\fs\,:\,P_{\ti{SW}}(w)=T,\,P(w^{-1})=V\}\eeq
for any $V\in\mb{SYT}(\mu)$. Formula (\ref{ps}) means that the cardinality of $\fs(T,V)$ only depends on the shapes of $T$ and $V$.

On the other hand, applying $\pi'$ to the multiplication formula (\ref{pr'm}), we get the shifted Littlewood-Richardson rule (See also \cite[Cor. 1.15]{Ser}): fix any $T\in\mb{ShSYT}(\mu)$, then
\[P_\la P_\mu=\sum_\nu b_{\la,\,\mu}^\nu P_\nu,\]
where $b_{\la,\,\mu}^\nu=|\{S\in\mb{ShSYT}(\nu/\la)\,:\, \mb{rect}(S)=T\}|$.

Moreover, we have
\begin{cor}
For any peak set $P$ in $[n-1]$ and $\la\vdash n$,
\begin{align}
&\label{pi}\phi(\Pi_P)=\sum\limits_{T\in\ti{ShSYT}_n\atop\ti{Peak}(T)=P}
\pi'(scl(T))=\sum\limits_{\la\vdash n\ti{ strict}}
|\{T\in\mb{ShSYT}(\la)\,:\,\mb{Peak}(T)=P\}|P_\la,\\
&\label{SK}S_\la=\te(s_\la)=\sum_{T\in\bs{SYT}(\la)}K_{\bs{Peak}(T)}.
\end{align}
\end{cor}
Note that formula (\ref{pi}) is equivalent to (\ref{QK}). Indeed,
\[[\Pi_P,Q_\la]=[\phi(\Pi_P),Q_\la]=|\{T\in\mb{ShSYT}(\la)\,:\,\mb{Peak}(T)=P\}|.\]
In particular, let $\emptyset_n:=\emptyset$ be the empty peak set in $[n-1]$, then $\phi(\Pi_{\emptyset_n})=P_{(n)}$. In fact, $\xy 0;/r.15pc/:
(0,2.5)*{};(20,2.5)*{}**\dir{-};
(0,-2.5)*{};(20,-2.5)*{}**\dir{-};
(0,2.5)*{};(0,-2.5)*{}**\dir{-};
(5,2.5)*{};(5,-2.5)*{}**\dir{-};
(15,2.5)*{};(15,-2.5)*{}**\dir{-};
(20,2.5)*{};(20,-2.5)*{}**\dir{-};
(2.5,0)*{\stt{1}};(10,0)*{\dots};
(17.5,0)*{\stt{n}};
\endxy$ is the unique standard shifted tableau with empty peak set.
It can be seen by Lemma \ref{des}, (2), as any permutation $w\in\fs_n$ satisfies $w_1^{-1}>\cdots>w_i^{-1}<\cdots<w_n^{-1}$ if $\mb{Peak}(w^{-1})=\emptyset_n$, thus $P_{\bs{SW}}(w)$ is of that form.

Combining formulas (\ref{kf}), (\ref{pf}) and (\ref{QK}), we have the following identity:
\[|\{S\in\mb{ShSYT}^\pm(\la)\,:\,\mb{Des}(S)=D\}|
=\sum_{T\in\ti{ShSYT}(\la)\atop\ti{Peak}(T)\subseteq
D\triangle(D+1)}2^{|\ti{Peak}(T)|+1-\ell(\la)}\]
for any strict partition $\la$ of $n$ and $D\subseteq[n-1]$. Now we give a stronger result.
\begin{prop}\label{dp}
Given $T\in\mb{ShSTY}(\la)$ for a strict partition $\la$ of $n$ and $D\subseteq[n-1]$ such that $\mb{Peak}(T)\subseteq
D\triangle(D+1)$, then
\[|\{S\in\mb{ShSYT}^\pm(\la)\,:\,\mb{Des}(S)=D,|S|=T\}|
=2^{|\ti{Peak}(T)|+1-\ell(\la)}.\]
\end{prop}
\begin{proof}
Here we give a bijective proof. Given a marked-standard shifted tableau $S$, first note that $\mb{Peak}(|S|)\subseteq\mb{Des}(S)\triangle(\mb{Des}(S)+1)$ (see \cite[Prop. 3.4]{As}). Indeed, $i\in\mb{Peak}(|S|)$ if and only if $i-1,i,i+1$ occur in $w(|S|)$ in the order $i-1,i+1,i$ or $i+1,i-1,i$. If such $i$ is primed in $S$, then $i-1\in\mb{Des}(S)$ and $i\notin\mb{Des}(S)$. If such $i$ is unprimed in $S$, then $i-1\notin\mb{Des}(S)$ and $i\in\mb{Des}(S)$.

Conversely, for any $T\in\mb{ShSYT}(\la)$ such that $\mb{Peak}(T)\subseteq D\triangle(D+1)$, it is easy to construct an $S\in\mb{ShSYT}^\pm(\la)$ satisfying $\mb{Des}(S)=D$ and $|S|=T$.
We only need to prove that there exist exactly $2^{|\ti{Peak}(T)|+1-\ell(\la)}$ such marked-standard shifted tableaux $S$'s. Suppose that $1=i_1<\cdots<i_k\leq n,\,k=\ell(\la)$, be the entries on the main diagonal of $T$, then it is easy to see that for any $j=1,\cdots,k-1$, $p_j:=|\mb{Peak}(T)\cap[i_j,i_{j+1}]|\geq1$, thus $|\mb{Peak}(T)|+1-\ell(\la)\geq0$. On the other hand, for any $S\in\mb{ShSYT}^\pm(\la)$ and an entry $i$ off the main diagonal of $S$, the change of the marking of $i$ does not affect the descent set of $S$ if and only if one of the following two cases holds:

(1) $1<i<n$ and $i-1,i,i+1$ occur in $w(|S|)$ in the order $i,i+1,i-1$ or $i,i-1,i+1$.

(2) $i=n$ and $n-1,n$ occur in $w(|S|)$ in the order $n,n-1$.

In order to get the desired result, one only needs to see that for any $j=1,\cdots,k-1$, there exist exactly $p_j-1$ entries in $[i_j,i_{j+1}]$ whose markings do not affect the descent set of $S$. If let $p_k:=|\mb{Peak}(T)\cap[i_k,n]|$, then there exist exactly $p_k$ such entries in $[i_k,n]$.
\end{proof}

\begin{exam}
For $\la=432,\,D=\{2,3,5,8\}$ and  $T=\raisebox{1.6em}{\xy 0;/r.16pc/:
(0,0)*{};(20,0)*{}**\dir{-};
(0,-5)*{};(20,-5)*{}**\dir{-};
(5,-10)*{};(20,-10)*{}**\dir{-};
(10,-15)*{};(20,-15)*{}**\dir{-};
(0,0)*{};(0,-5)*{}**\dir{-};
(5,0)*{};(5,-10)*{}**\dir{-};
(10,0)*{};(10,-15)*{}**\dir{-};
(15,0)*{};(15,-15)*{}**\dir{-};
(20,0)*{};(20,-15)*{}**\dir{-};
(2.5,-2.5)*{\stt{1}};(7.5,-2.5)*{\stt{2}};(12.5,-2.5)*{\stt{4}};
(17.5,-2.5)*{\stt{6}};(7.5,-7.5)*{\stt{3}};(12.5,-7.5)*{\stt{5}};
(17.5,-7.5)*{\stt{8}};(12.5,-12.5)*{\stt{7}};(17.5,-12.5)*{\stt{9}};
\endxy}~$, then $\mb{Peak}(T)=\{2,4,6,8\}$, and we have those $S\in\mb{ShSYT}^\pm(\la)$ satisfying $\mb{Des}(S)=D$ and $|S|=T$ as follows:
\[\raisebox{1.6em}{\xy 0;/r.16pc/:
(0,0)*{};(20,0)*{}**\dir{-};
(0,-5)*{};(20,-5)*{}**\dir{-};
(5,-10)*{};(20,-10)*{}**\dir{-};
(10,-15)*{};(20,-15)*{}**\dir{-};
(0,0)*{};(0,-5)*{}**\dir{-};
(5,0)*{};(5,-10)*{}**\dir{-};
(10,0)*{};(10,-15)*{}**\dir{-};
(15,0)*{};(15,-15)*{}**\dir{-};
(20,0)*{};(20,-15)*{}**\dir{-};
(2.5,-2.5)*{\stt{1}};(7.5,-2.5)*{\stt{2}};(12.5,-2.5)*{\stt{4'}};
(17.5,-2.5)*{\stt{6'}};(7.5,-7.5)*{\stt{3}};(12.5,-7.5)*{\stt{5}};
(17.5,-7.5)*{\stt{8}};(12.5,-12.5)*{\stt{7}};(17.5,-12.5)*{\stt{9}};
\endxy}~,\,
\raisebox{1.6em}{\xy 0;/r.16pc/:
(0,0)*{};(20,0)*{}**\dir{-};
(0,-5)*{};(20,-5)*{}**\dir{-};
(5,-10)*{};(20,-10)*{}**\dir{-};
(10,-15)*{};(20,-15)*{}**\dir{-};
(0,0)*{};(0,-5)*{}**\dir{-};
(5,0)*{};(5,-10)*{}**\dir{-};
(10,0)*{};(10,-15)*{}**\dir{-};
(15,0)*{};(15,-15)*{}**\dir{-};
(20,0)*{};(20,-15)*{}**\dir{-};
(2.5,-2.5)*{\stt{1}};(7.5,-2.5)*{\stt{2}};(12.5,-2.5)*{\stt{4'}};
(17.5,-2.5)*{\stt{6'}};(7.5,-7.5)*{\stt{3}};(12.5,-7.5)*{\stt{5'}};
(17.5,-7.5)*{\stt{8}};(12.5,-12.5)*{\stt{7}};(17.5,-12.5)*{\stt{9}};
\endxy}~,\,
\raisebox{1.6em}{\xy 0;/r.16pc/:
(0,0)*{};(20,0)*{}**\dir{-};
(0,-5)*{};(20,-5)*{}**\dir{-};
(5,-10)*{};(20,-10)*{}**\dir{-};
(10,-15)*{};(20,-15)*{}**\dir{-};
(0,0)*{};(0,-5)*{}**\dir{-};
(5,0)*{};(5,-10)*{}**\dir{-};
(10,0)*{};(10,-15)*{}**\dir{-};
(15,0)*{};(15,-15)*{}**\dir{-};
(20,0)*{};(20,-15)*{}**\dir{-};
(2.5,-2.5)*{\stt{1}};(7.5,-2.5)*{\stt{2}};(12.5,-2.5)*{\stt{4'}};
(17.5,-2.5)*{\stt{6'}};(7.5,-7.5)*{\stt{3}};(12.5,-7.5)*{\stt{5}};
(17.5,-7.5)*{\stt{8}};(12.5,-12.5)*{\stt{7}};(17.5,-12.5)*{\stt{9'}};
\endxy}~,\,
\raisebox{1.6em}{\xy 0;/r.16pc/:
(0,0)*{};(20,0)*{}**\dir{-};
(0,-5)*{};(20,-5)*{}**\dir{-};
(5,-10)*{};(20,-10)*{}**\dir{-};
(10,-15)*{};(20,-15)*{}**\dir{-};
(0,0)*{};(0,-5)*{}**\dir{-};
(5,0)*{};(5,-10)*{}**\dir{-};
(10,0)*{};(10,-15)*{}**\dir{-};
(15,0)*{};(15,-15)*{}**\dir{-};
(20,0)*{};(20,-15)*{}**\dir{-};
(2.5,-2.5)*{\stt{1}};(7.5,-2.5)*{\stt{2}};(12.5,-2.5)*{\stt{4'}};
(17.5,-2.5)*{\stt{6'}};(7.5,-7.5)*{\stt{3}};(12.5,-7.5)*{\stt{5'}};
(17.5,-7.5)*{\stt{8}};(12.5,-12.5)*{\stt{7}};(17.5,-12.5)*{\stt{9'}};
\endxy}~.\]
\end{exam}

By Prop. \ref{dp}, we also have the following formula translating the result in \cite[Th. 3.6]{Ste}:
\begin{cor}
For any $T\in\mb{ShSYT}(\la)$,
\beq\label{kf1}
K_{\bs{Peak}(T)}=2^{\ell(\la)}
\sum_{S\in\ti{ShSYT}^\pm(\la)\atop|S|=T}F_{c(S)}.\eeq
\end{cor}

\begin{rem}
From \cite[Ch. III, \S 8, Ex. 7]{Mac}, we know that the modified Schur function
\[S_\la=\sum_{T\in\bs{SSYT}^\pm(\la)}x^{\bs{wt}(T)},\]
where $\mb{SSYT}^\pm(\la)$ is the set of (unshifted) marked tableaux of shape $\la$ restricted by the first three conditions for $\mb{ShSSYT}^\pm$. If we further define the (unshifted) marked-standard tableaux, denoted by $\mb{SYT}^\pm$, and their descent sets, then we have
\[S_\la=\sum_{T\in\bs{SYT}^\pm(\la)}F_{c(T)},\]
analogous to Prop. \ref{bij}. Similarly, one can also get the unshifted counterpart of Prop. \ref{dp}: Given $T\in\mb{SYT}(\la)$ for one $\la\vdash n$ and $D\subseteq[n-1]$ such that $\mb{Peak}(T)\subseteq D\triangle(D+1)$, we have
\[|\{S\in\mb{SYT}^\pm(\la)\,:\,\mb{Des}(S)=D,|S|=T\}|
=2^{|\ti{Peak}(T)|+1}.\]
\end{rem}

Finally, we raise some further questions to consider.

(i) We wonder whether there exists an algebra epimorphism $\ms{PR}'\twoheadrightarrow \ms{SPR}'$ extending $\Te$ and lifting $\te$, which means the following commutative diagram
\begin{equation}\label{dia'}
\raisebox{2em}{\xymatrix@H=1.5em{
\ms{NSym }\ar@{>->}[r]\ar@{>>}[d]^-{\Te}& \ms{PR}' \ar@{>>}[r]^-{\pi'}\ar@{.>>}[d]^-?&\La\ar@{>>}[d]^-\te\\
\ms{Peak }\ar@{>->}[r]& \ms{SPR}' \ar@{>>}[r]^-{\pi'}&\Om
}}
\end{equation}
or dually if there exists a coalgebra monomorphism $\ms{SPR}\rightarrowtail\ms{PR}$ such that
\begin{equation}\label{dia1'}
\raisebox{2em}{\xymatrix@H=1.5em{
{\Om\,\,}\ar@{>->}[r]\ar@{>->}[d]& \ms{SPR}\ar@{>>}[r]\ar@{>.>}[d]^-?&\ms{Peak}^* \ar@{>->}[d]\\
{\La\,}\ar@{>->}[r]&
\ms{PR }\ar@{>>}[r]&\ms{QSym}
}}.
\end{equation}

In fact, we find one candidate as follows. For any strict $\la$, let $T_\la$ be the unique shifted standard tableau with $\la_{i-1}+1,\la_{i-1}+2,\dots,\la_i$ lying in the $i$th row for $i=1,\dots,\ell(\la)$. Define
\[j:\ms{SPR}\rw\ms{PR},\,\lan T\ran\mapsto2^{\ell(\la)}\sum_{w\in\fs(T)}[P(w)]\]
for any $T\in\mb{ShSYT}(\la)$, where $\fs(T):=\{w\in\fs\,:\,P_\bs{SW}(w^{-1})=T_\la,\,
\left|Q_\bs{SW}(w^{-1})\right|=T\}$. Note that $P(w)$ is only determined by $Q_\bs{SW}(w^{-1})=P_\bs{mix}(w)$
according to Remark \ref{rem}, thus $T_\la$
can be replaced by any fixed $T'\in\mb{ShSYT}(\la)$ to give the same map $j$. Also, $|\fs(T)|=2^{|\la|-\ell(\la)}$.

From (\ref{ps}), we know that the embedding $\Om\rw\La\rw\ms{PR}$ maps $Q_\la$ to
\[\begin{split}
&2^{\ell(\la)}\sum\limits_\mu
|\{U\in\mb{SYT}(\mu)\,:\,P_{\ti{SW}}(w(U))=T_\la\}|\lb
\sum_{T\in\bs{SYT}(\mu)}[T]\rb\\
&=2^{\ell(\la)}\sum\limits_{w\in\fs\,:\,P_{\ti{SW}}(w(Q(w)))=T_\la}[P(w)]
=2^{\ell(\la)}\sum\limits_{w\in\fs\,:\,P_{\ti{SW}}(w^{-1})=T_\la}[P(w)]
\end{split}\]
On the other hand,
\begin{align*}
j\lb\sum\nolimits_{T\in\bs{ShSYT}(\la)}\lan T\ran\rb &=
2^{\ell(\la)}\sum_{T\in\bs{ShSYT}(\la)}\sum_{w\in\fs(T)}[P(w)]\\
&=2^{\ell(\la)}\sum\limits_{w\in\fs\,:\,P_{\ti{SW}}(w^{-1})=T_\la}[P(w)].
\end{align*}
This is the left commutative square in diagram (\ref{dia1'}).

Meanwhile, the image of $\lan T\ran,\,T\in\mb{ShSYT}(\la)$, under the map
$\ms{SPR}\rw\ms{PR}\rw\ms{QSym}$ is
\[
2^{\ell(\la)}\sum_{w\in\fs(T)}F_{c(P(w))}
=2^{\ell(\la)}\sum_{w\in\fs(T)}F_{c(Q_{\ti{SW}}(w^{-1}))}
=2^{\ell(\la)}\sum_{S\in\ti{ShSYT}^\pm(\la)\atop|S|=T}
F_{c(S)}=K_{\bs{Peak}(T)}
\]
by formula (\ref{kf1}), thus it coincides with that of $\ms{SPR}\rw\ms{Peak}^*\rw\ms{QSym}$, and the right commutative square in diagram (\ref{dia1'}) also holds.

Dually, we define
\[\Xi:\ms{PR}'\rw\ms{SPR}',\,cl(U)\mapsto
\sum_{\la\bs{ strict}} 2^{\ell(\la)}\sum_{w\in\fs(\la,U)}scl\lb\left|
Q_\bs{SW}(w)\right|\rb\]
for any standard tableau $U$, where $\fs(\la,U):=\fs(T_\la,U)$ defined in (\ref{set}). It makes
\[\lan \lan T\ran,\Xi(cl(U))\ran=\lan j(\lan T\ran),cl(U)\ran.\]

The image of $R_\al,\,\al\vDash n$, under the map
$\ms{NSym}\stackrel{\Te}{\rw}\ms{Peak}\rw\ms{SPR}'$ is
\[\sum_{P\ti{ peak set in } [n-1]
\atop P\subseteq D(\al)\triangle(D(\al)+1)}2^{|P|+1}\lb
\sum_{T\in\ti{ShSYT}_n\atop\ti{Peak}(T)=P}scl(T)\rb
=\sum_{T\in\ti{ShSYT}_n
\atop \ti{Peak}(T)\subseteq D(\al)\triangle(D(\al)+1)}2^{|\ti{Peak}(T)|+1}scl(T).\]
On the other hand, its image under the map
$\ms{NSym}\rw\ms{PR}'\stackrel{\Xi}{\rw}\ms{SPR}'$ is
\[\begin{split}
\Xi\lb\sum_{c(U)=\al}cl(U)\rb&=\sum_{\la\bs{ strict}} 2^{\ell(\la)}\sum_{c(U)=\al,\,w\in\fs(\la,U)}scl\lb\left|
Q_\bs{SW}(w)\right|\rb\\
&=\sum_{\la\bs{ strict}} 2^{\ell(\la)}\sum_{w\in\fs(\la,\cdot),\,c(w)=\al}scl\lb\left|
Q_\bs{SW}(w)\right|\rb\\
&=\sum_{\la\bs{ strict}} 2^{\ell(\la)}\sum_{S\in\ti{ShSYT}^\pm(\la)
\atop c(S)=\al}scl(|S|).
\end{split}\]
Hence, the left commutative square in diagram (\ref{dia'}) holds by Prop. \ref{dp}. Finally, note that
\[\te\circ\pi'(cl(U))=\te(s_\la)=S_\la\]
for any $U\in\mb{SYT}(\la)$. Meanwhile,
\begin{align*}
&\pi'\circ\Xi(cl(U))=\sum_{\mu\bs{ strict}} 2^{\ell(\mu)}\sum_{w\in\fs(\mu,U)}\pi'\lb scl\lb\left|
Q_\bs{SW}(w)\right|\rb\rb\\
&=\sum_{\mu\bs{ strict}}|\fs(\mu,U)|Q_\mu
=\sum_{\mu\bs{ strict}}|\fs(\mu,U)|\lb\sum_{T\in\bs{ShSYT}(\mu)}K_{\bs{Peak}(T)}\rb\\
&=\sum_{\mu\bs{ strict}}\lb\sum_{T\in\bs{ShSYT}(\mu)}|\fs(T,U)|K_{\bs{Peak}(T)}\rb
=\sum_{T\in\bs{ShSYT}}\lb\sum_{w\in\fs(T,U)}K_{\bs{Peak}(w^{-1})}\rb\\
&=\sum_{P(w^{-1})=U}K_{\bs{Peak}(Q(w^{-1}))}
=\sum_{V\in\bs{SYT}(\la)}K_{\bs{Peak}(V)}
=S_\la
\end{align*}
by formula (\ref{SK}), thus the right commutative square in diagram (\ref{dia'}) also holds. In particular, we get a somewhat new expansion formula
\beq\label{SQ}
S_\la=\sum_{\mu\bs{ strict}}|\fs(\mu,U)|Q_\mu\eeq
for any fixed $U\in\mb{SYT}(\la)$. In the end, one should see that all the adjoint pairs in (\ref{bil}) have been lifted onto the level of (shifted) PR algebras. Unfortunately, we find that $j$ \textit{fails} to be a coalgebra homomorphism, and thus $\Xi$ is also not an algebra homomorphism. So far the authors have no idea how to modify the definition of $j$ to rescue the situation.

One simple counterexample is as follows.
\begin{exam}
For $T=\raisebox{1.6em}{\xy 0;/r.16pc/:
(0,0)*{};(20,0)*{}**\dir{-};
(0,-5)*{};(20,-5)*{}**\dir{-};
(5,-10)*{};(15,-10)*{}**\dir{-};
(10,-15)*{};(15,-15)*{}**\dir{-};
(0,0)*{};(0,-5)*{}**\dir{-};
(5,0)*{};(5,-10)*{}**\dir{-};
(10,0)*{};(10,-15)*{}**\dir{-};
(15,0)*{};(15,-15)*{}**\dir{-};
(20,0)*{};(20,-5)*{}**\dir{-};
(2.5,-2.5)*{\stt{1}};(7.5,-2.5)*{\stt{2}};(12.5,-2.5)*{\stt{3}};
(17.5,-2.5)*{\stt{6}};
(7.5,-7.5)*{\stt{4}};(12.5,-7.5)*{\stt{5}};(12.5,-12.5)*{\stt{7}};
\endxy}~$, one can check that
\[P_{\bs{SW}}(5172364)=\raisebox{1.6em}{\xy 0;/r.16pc/:
(0,0)*{};(20,0)*{}**\dir{-};
(0,-5)*{};(20,-5)*{}**\dir{-};
(5,-10)*{};(15,-10)*{}**\dir{-};
(10,-15)*{};(15,-15)*{}**\dir{-};
(0,0)*{};(0,-5)*{}**\dir{-};
(5,0)*{};(5,-10)*{}**\dir{-};
(10,0)*{};(10,-15)*{}**\dir{-};
(15,0)*{};(15,-15)*{}**\dir{-};
(20,0)*{};(20,-5)*{}**\dir{-};
(2.5,-2.5)*{\stt{1}};(7.5,-2.5)*{\stt{2}};(12.5,-2.5)*{\stt{3}};
(17.5,-2.5)*{\stt{4}};
(7.5,-7.5)*{\stt{5}};(12.5,-7.5)*{\stt{6}};(12.5,-12.5)*{\stt{7}};
\endxy}~,\,
Q_{\bs{SW}}(5172364)=\raisebox{1.6em}{\xy 0;/r.16pc/:
(0,0)*{};(20,0)*{}**\dir{-};
(0,-5)*{};(20,-5)*{}**\dir{-};
(5,-10)*{};(15,-10)*{}**\dir{-};
(10,-15)*{};(15,-15)*{}**\dir{-};
(0,0)*{};(0,-5)*{}**\dir{-};
(5,0)*{};(5,-10)*{}**\dir{-};
(10,0)*{};(10,-15)*{}**\dir{-};
(15,0)*{};(15,-15)*{}**\dir{-};
(20,0)*{};(20,-5)*{}**\dir{-};
(2.5,-2.5)*{\stt{1}};(7.5,-2.5)*{\stt{2'}};(12.5,-2.5)*{\stt{3}};
(17.5,-2.5)*{\stt{6}};
(7.5,-7.5)*{\stt{4}};(12.5,-7.5)*{\stt{5}};(12.5,-12.5)*{\stt{7}};
\endxy}~,\]
thus $w:=2457163=(5172364)^{-1}\in\fs(T)$. Now fix $i=2$, then $u:=w|_{[1,2]}=21,\,v:=\mb{st}(w|_{[3,7]})=23541$,
and $P(v)=\raisebox{1.6em}{\xy 0;/r.16pc/:
(0,0)*{};(15,0)*{}**\dir{-};
(0,-5)*{};(15,-5)*{}**\dir{-};
(0,-10)*{};(5,-10)*{}**\dir{-};
(0,-15)*{};(5,-15)*{}**\dir{-};
(0,0)*{};(0,-15)*{}**\dir{-};
(5,0)*{};(5,-15)*{}**\dir{-};
(10,0)*{};(10,-5)*{}**\dir{-};
(15,0)*{};(15,-5)*{}**\dir{-};
(2.5,-2.5)*{\stt{1}};(7.5,-2.5)*{\stt{3}};(12.5,-2.5)*{\stt{4}};
(2.5,-7.5)*{\stt{2}};(2.5,-12.5)*{\stt{5}};
\endxy}~$.

On the other hand, $T=(T')_S$, where
\[T'=\raisebox{.8em}{\xy 0;/r.16pc/:
(0,0)*{};(10,0)*{}**\dir{-};
(0,-5)*{};(10,-5)*{}**\dir{-};
(0,0)*{};(0,-5)*{}**\dir{-};
(5,0)*{};(5,-5)*{}**\dir{-};
(10,0)*{};(10,-5)*{}**\dir{-};
(2.5,-2.5)*{\stt{1}};(7.5,-2.5)*{\stt{2}};
\endxy}\,,\,S=\raisebox{1.6em}{\xy 0;/r.16pc/:
(5,0)*{};(15,0)*{}**\dir{-};
(0,-5)*{};(15,-5)*{}**\dir{-};
(0,-10)*{};(10,-10)*{}**\dir{-};
(5,-15)*{};(10,-15)*{}**\dir{-};
(0,-5)*{};(0,-10)*{}**\dir{-};
(5,0)*{};(5,-15)*{}**\dir{-};
(10,0)*{};(10,-15)*{}**\dir{-};
(15,0)*{};(15,-5)*{}**\dir{-};
(7.5,-2.5)*{\stt{1}};(12.5,-2.5)*{\stt{4}};(2.5,-7.5)*{\stt{2}};
(7.5,-7.5)*{\stt{3}};(7.5,-12.5)*{\stt{5}};
\endxy}\,,
\,T'':=\mb{rect}(S)=\raisebox{1.2em}{\xy 0;/r.16pc/:
(0,0)*{};(20,0)*{}**\dir{-};
(0,-5)*{};(20,-5)*{}**\dir{-};
(5,-10)*{};(10,-10)*{}**\dir{-};
(0,0)*{};(0,-5)*{}**\dir{-};
(5,0)*{};(5,-10)*{}**\dir{-};
(10,0)*{};(10,-10)*{}**\dir{-};
(15,0)*{};(15,-5)*{}**\dir{-};
(20,0)*{};(20,-5)*{}**\dir{-};
(2.5,-2.5)*{\stt{1}};(7.5,-2.5)*{\stt{2}};
(12.5,-2.5)*{\stt{3}};(17.5,-2.5)*{\stt{4}};(7.5,-7.5)*{\stt{5}};
\endxy}~.\]
Then $\fs(T'')=\{12354,21354,31254,32154,41253,42153,43152,43251\}$. Hence, $P(v)$ $\notin\{P(w')\,:\,w'\in\fs(T'')\}$, which means that
\[\De\circ j(\lan T\ran)\neq(j\ot j)\circ\De(\lan T\ran).\]
\end{exam}

(ii) In \cite{Li} one of us 
has defined the $q$-analogue of PR algebras and applied that to study the odd Schur functions introduced by Ellis, Khovanov and Lauda.
We expect there should also exist an odd counterpart for Schur's P-functions. Indeed, if one uses the following odd shifted Knuth transformations,

(OSK1)\quad$xzy\sim -zxy$ if $x<y<z$,

(OSK2)\quad$yxz\sim -yzx$ if $x<y<z$,

(OSK3)\quad$xy\sim -yx$ if $x,y$ are the first two letters of the permutation,

\noindent
 one can then define the odd part of the shifted PR algebras. However one still needs to find an odd analog of $\Om$ inside the algebra of odd symmetric functions in search of the odd Schur P-functions.

\bigskip

\centerline{\bf Acknowledgments} NJ
acknowledges the partial support of Simons Foundation grant 198129
and NSFC grant 11271138.

\bigskip
\bibliographystyle{amsalpha}

\end{document}